\documentclass[a4paper, 11pt]{amsart}

\usepackage{amsfonts}
\usepackage{amsfonts}
\usepackage{amssymb}
\usepackage{amsthm}
\usepackage{amssymb}

\usepackage{amsmath}
\usepackage{latexsym}
\usepackage{epsfig}
\usepackage[mathscr]{eucal}
\usepackage{color}

\newtheorem{theorem}{Theorem}
\newtheorem{lemma}[theorem]{Lemma}

\newtheorem{definition}[theorem]{Definition}
\newtheorem{example}[theorem]{Example}

\newtheorem{remark}[theorem]{Remark}

\begin{document}
\title[Stochastic stability of galactic dynamo]{Stochastic stability analysis of a reduced galactic dynamo model with perturbed $\alpha$-effect}

\author{C\'onall Kelly}
\email{conall.kelly@uwimona.edu.jm}
\address{Department of Mathematics, The University of the West Indies, Mona, Kingston 7, Jamaica. }


\date{\today}

\begin{abstract}
We investigate the asymptotic behaviour of a reduced $\alpha\Omega$-dynamo model of magnetic field generation in spiral galaxies where fluctuation in the $\alpha$-effect results in a system with state-dependent stochastic perturbations.

By computing the upper Lyapunov exponent of the linearised model, we can identify regions of instability and stability in probability for the equilibrium of the nonlinear model; in this case the equilibrium solution corresponds to a magnetic field that has undergone catastrophic quenching. These regions are compared to regions of exponential mean-square stability and regions of sub- and super-criticality in the unperturbed linearised model. Prior analysis in the literature which focuses on these latter regions does not adequately address the corresponding transition in the nonlinear stochastic model.

Finally we provide a visual representation of the influence of drift non-normality and perturbation intensity on these regions. 

\end{abstract}


\maketitle

\section{Introduction}
Magnetic fields in large scale astrophysical objects may be generated by a turbulent dynamo, potentially driven by the $\alpha$-effect, which arises in this context when differential rotation in the underlying turbulent flow distorts magnetic field lines in such a way that the magnetic field is amplified. However, this effect is subject to quenching, and there is an ongoing investigation in the literature (see \cite{SriSin,SurSub} for a recent review) examining the question of how random fluctuation of the $\alpha$-effect and/or extrinsic random perturbation may overcome this quenching and allow for the spontaneous generation of large-scale magnetic fields. To the best of the author's knowledge, Farrell \& Ioannou~\cite{FI1999} were the first to suggest the combined role of operator non-normality and stochastic forcing in the generation of magnetic fields in stellar and galactic objects.

In this article we build upon the analysis of Fedotov et al~\cite{Fed,FBR,FBR2}, who explored the role played by the interaction of non-normality with stochastic perturbation using a reduced stochastic differential equation (SDE) $\alpha\Omega$-dynamo model based on a no-$z$ model for spiral galaxies due to Moss~\cite{Moss}. This model assumes a thin rotating turbulent disk of a conducting medium of uniform thickness, where the angular velocity varies by distance from the centre. The sensitivity of the reduced dynamo model to perturbation is illustrated as a consequence of the non-normality of the unperturbed system. A normal matrix is one that is diagonalisable by a unitary transform. Solutions of a system of first-order ordinary differential equations governed by a non-normal coefficient matrix undergo transients that are potentially large in extent and duration, even though the spectral abscissa of the coefficient may be negative. These transients are amplified by perturbation, and consequently the stability of an equilibrium solution may be vulnerable either to stochastic perturbation, or perturbation by higher-order terms omitted in the process of linearisation, even if the intensity is small. The same phenomenon is discussed in the context of test system selection for numerical methods of SDEs in \cite{BuKe11}, and some further applied examples are treated in \cite{BuKe14}.

We are interested in the instability and stability in probability of the equilibrium solution that corresponds to a magnetic field that has undergone catastrophic quenching. The goal of the article is to determine the role of multiplicative stochastic perturbation arising from random fluctuation in the $\alpha$-effect in the absence of additive noise. To that end we derive a formula for the upper Lyapunov exponent of the linearised system in terms of the model parameters and observe how the sign responds to changes in the non-normality of the drift and the intensity of the perturbation.

In \cite{FBR}, the authors use the linearisation of their model to show that the interaction of the multiplicative stochastic term (random fluctuation in the $\alpha$-effect) with the non-normality of the unperturbed system leads to a loss of mean-square asymptotic stability in the subcritical case, interpreted as the generation of a sustained magnetic field. A preliminary numerical investigation of trajectories of the nonlinear model yielded observations of a stochastically induced phase transition in the presence of additive noise. 

In \cite{FBR2} the authors continued their theoretical investigation, using adiabatic elimination to derive a stochastic differential equation that governs the behaviour of the slow variable of the nonlinear system close to the transition. The probability density function of this governing SDE underwent qualitative changes in response to an increase in intensity of the multiplicative noise term.

To see why we work with Lyapunov exponents, consider that the work in \cite{Fed,FBR,FBR2} addresses the exponential mean-square stability of the linearised equation in the drift-subcritical regime, and the behaviour of the distribution of the slowly varying component of the nonlinear system. Exponential mean-square stability is sufficient but not necessary for stability in probability of the equilibrium of the nonlinear stochastic system. Indeed, there is currently no theoretical link whatsoever between linear mean-square instability and nonlinear instability. In order to draw conclusions about the development of instability of any kind in the nonlinear model via the linearisation an understanding of pathwise behaviour beyond that provided by numerical simulation is necessary and one must compute upper Lyapunov exponents. Such an analysis of the stability of the nonlinear stochastic model has never been carried out and we do so here, comparing regions of stability provided by the upper Lyapunov exponent to those provided by exponential mean-square stability in in \cite{FBR}, and to regions of sub/super-criticality in the unperturbed model.

The structure of the article is as follows. In Section \ref{background} we provide the model equations and the appropriate definitions and results from the theory of stochastic stability. We select and describe a suitable measure of non-normality for the linearised model system, present a sharp condition for its exponential mean-square stability, and outline the theory of Lyapunov exponents for linear stochastic systems. In Section \ref{asasy} we use the methods of Imkeller \& Lederer~\cite{IL99,IL}, which build upon the work of Arnold et al~\cite{Arnold1,Arnold2,Arnold3}, to derive a formula for the upper Lyapunov exponent $\lambda$ of the linearised model, expressed in terms of gamma and generalised hypergeometric functions. In Section \ref{visualisation} we use the results presented in Sections \ref{background} and \ref{asasy} to plot regions of exponential mean-square stability along with regions of positive and negative $\lambda$: demonstrating the influence of perturbation of the $\alpha$-effect, showing how these regions change with the non-normality of the drift, and identifying regions of the parameter space where the equilibrium of the nonlinear model system is stable or unstable in probability. 

\section{Background: the model system, stochastic stability and non-normality}\label{background}

\subsection{A reduced $\alpha\Omega$-dynamo model}
We select the following reduced 2-dimensional stochastic model of an $\alpha\Omega$-dynamo, which was used by Fedotov et al~\cite{FBR,FBR2} to investigate the role of stochastic perturbation in the generation of large-scale magnetic fields in spiral galaxies. The radial ($B_r$) and azimuthal ($B_\varphi$) components of the magnetic field $\mathbf{B}=(B_r,B_\varphi)^T$ are governed by
\begin{equation}\label{eq:nonlin}
d\mathbf{B}=f_1(\mathbf{B})dt+f_2(\mathbf{B})dW_1(t),\quad t\geq 0,
\end{equation}
where $W_1$ is a Wiener process and the drift and diffusion coefficients $f_1$ and $f_2$ are
\[
f_1(B_r,B_{\varphi})=\begin{pmatrix}-(\delta\varphi_\alpha(B_\varphi)B_\varphi+\varepsilon\varphi_\beta(B_\varphi)B_r)\\
 -(gB_r+\varepsilon\varphi_\beta(B_\varphi)B_\varphi)
\end{pmatrix};\quad f_2(B_r,B_\varphi)=\begin{pmatrix}-\sqrt{2\sigma_1}\varphi_\alpha(B_\varphi)B_\varphi\\ 0\end{pmatrix}.
\]
Here, $g,\delta,\varepsilon\in\mathbb{R}^+$ are dimensionless parameters with typical values $g\approx 1$, $\delta,\varepsilon\in[0.01,0.1]$, and $\sigma_{1}>0$. $\varphi_\alpha(\cdot)$ and $\varphi_\beta(\cdot)$ are nonlinear quenching functions associated with $\alpha$ and $\beta$ respectively:
\[
\varphi_\alpha(B_\varphi)=\frac{1}{1+k_\alpha B_\varphi^2};\qquad\varphi_\beta(B_\varphi)=\frac{1+B_\varphi^2}{1+(k_\beta+1)B_\varphi^2},
\]
where $k_\alpha,k_\beta$ are constants with order of magnitude $1$. Terms parameterised by $\alpha$ correspond to the contribution of the $\alpha$-effect on the field, and those parameterised by $\beta$ refer to the contribution of turbulent magnetic diffusivity. The relationship between the reduced model \eqref{eq:nonlin} and the full model is described and motivated in \cite{FBR}. Note that we omit additive noise from the model in order to focus on the influence of perturbations arising from stochastic fluctuation in the $\alpha$-effect.

The steady-state solutions of the unperturbed nonlinear system given by \eqref{eq:nonlin} with $\sigma_1=0$ satisfy
\begin{subequations}\label{eq:eq}
\begin{eqnarray}
B_\varphi&=&-\delta\varphi_\alpha(B_\varphi)B_\varphi+\frac{\varepsilon^2}{g}\varphi^2_\beta(B_\varphi)B_\varphi;\\
B_r&=&-\frac{\varepsilon}{g}B_\varphi\varphi_\beta(B_\varphi).
\end{eqnarray}
\end{subequations}
\begin{example}
With the parameter set $\varepsilon=0.1,\delta=0.01,g=0.99,k_\alpha=k_\beta=1$ from Section 3 of \cite{FBR}, there are five distinct steady state solutions, given to five significant digits by
\[
(B_r^\ast,B_\varphi^\ast)=(0,0),\,(\pm 0.01066, \pm 0.10154),\,(\pm 0.08246, \mp 1.2522)
\]
\end{example}

However, if $\sigma_1\neq 0$, it follows from \eqref{eq:nonlin} that the equilibrium solution $\mathbf{B}^\ast=(B_r^\ast,B_\varphi^\ast)=(0,0)$ is preserved by such perturbations, but any nonzero equilibrium would have to simultaneously satisfy \eqref{eq:eq} and $\sqrt{2\sigma_1}B_\varphi\varphi_\alpha(B_\varphi)=0$. This latter is impossible, since $\varphi_\alpha(x)=1/(1+k_\alpha x^2)>0$ for all $x\in\mathbb{R}$. So \eqref{eq:nonlin} has only a single equilibrium at $(0,0)$. Since $\varphi_\alpha(x)=1+O(x^2)$ and similarly $\varphi_\beta(x)=1+O(x^2)$, we may linearise about this equilibrium to obtain
\begin{equation}\label{EQ:LINB}
d\mathbf{B}=\Lambda\mathbf{B}dt+\Sigma\mathbf{B}dW_1(t),\quad t\geq 0,
\end{equation}
where
\[
\Lambda=\begin{pmatrix}
-\varepsilon & -\delta\\
-g & -\varepsilon
\end{pmatrix};\quad \Sigma=\begin{pmatrix}
0 & -\sqrt{2\sigma_1}\\
0 & 0
\end{pmatrix}
\]
are the drift and diffusion coefficient matrices.

\subsection{Characterising the departure from normality of the drift}\label{sec:scalnn}
We wish to characterise the departure from normality of the drift coefficient of \eqref{EQ:LINB}. A review of several scalar measures is provided in Trefethen \& Embree~\cite[Chapter 48]{Tref}, who note that the non-normality of a matrix is too complex to be fully characterised by a scalar measure. However, for the purposes of this article, the selection of the following scalar measure, due to Henrici~\cite{Hen}, suffices to illustrate the influence of parameter variation on the non-normality of the drift coefficient of \eqref{EQ:LINB}. 
\begin{definition}\label{def:Hnn}
The departure from normality of a square matrix $\Lambda\in\mathbb{C}^{N\times N}$ may be characterised as
\[
\text{dep}(\Lambda)=\min_{\substack{\Lambda=U(D+R)U^\ast\\ \text{(Schur decomposition)}}}\|R\|,
\]
where $U$ is a unitary matrix, $D$ is a diagonal matrix, and $R$ is a strictly upper triangular matrix. 
\end{definition}
In the Frobenius norm, this measure has computational form
\begin{equation}\label{eq:Hnn}
\text{dep}_F(\Lambda)=\sqrt{\|\Lambda\|^2_F-\|D\|^2_F}=\sqrt{\sum_{j=1}^{N}s_j^2-\sum_{j=1}^{N}|\mu_j|^2},
\end{equation}
where $\{s_j\}_{j=1}^N$ and $\{\mu_j\}_{j=1}^{N}$ are the singular values and eigenvalues of $\Lambda$, respectively. Recall that if $\Lambda=W\Delta V^\ast$ is the singular value decomposition of $A$, where $W$ and $V^\ast$ are unitary matrices and $\Delta$ is a diagonal matrix uniquely determined by $\Lambda$, then the singular values of $\Lambda$ are the diagonal entries of $\Delta$.
\begin{example}
If $\Lambda$ is the drift coefficient of \eqref{EQ:LINB}, the eigenvalues of $A$ are given by $
\mu_1=-\varepsilon\pm\sqrt{\delta g}$
and the singular values by
\begin{equation*}
s_{1,2}=\frac{1}{\sqrt{2}}\sqrt{2\varepsilon^2+\delta^2+g^2\pm\sqrt{((\varepsilon^2+\delta^2)-\varepsilon^2-g^2)^2+4\varepsilon^2(g+\delta)^2}}.
\end{equation*}
Therefore
\begin{eqnarray*}
\text{dep}_F(\Lambda)=\sqrt{s_1^2+s_2^2-|\mu_1|^2-|\mu_2|^2}&=&\sqrt{2\varepsilon^2+\delta^2+g^2-2(\varepsilon^2+\delta g)}\\
&=&|g-\delta|.
\end{eqnarray*}
We see that, since the diagonal entries are identical, the departure from normality under this measure corresponds to the absolute difference of the off-diagonal entries, and is independent of $\varepsilon$. Moreover, since $g\approx 1$ and $\delta\in[0.01,0.1]$, $\text{dep}_F(\Lambda)$ will typically be in the range $[0.9,0.99]$.
\end{example}

\subsection{The stability of equilibria of stochastic systems}
Khas'minskii~\cite{khas} maps the relationship between the stochastic stability of the equilibria of the linear approximation \eqref{EQ:LINB} and the nonlinear system \eqref{eq:nonlin}. First, note that
$\Lambda,\Sigma$ and $f_1,f_2$ satisfy
\[
\lim_{|\mathbf{x}|\to 0}\frac{|f_1(\mathbf{x})-\Lambda\mathbf{x}|+|f_2(\mathbf{x})-\Sigma\mathbf{x}|}{|\mathbf{x}|}=0.
\]
Next, we define various relevant notions of stochastic stability.
\begin{definition}\label{def:asstab}
The solution $\mathbf{X}(t)\equiv \mathbf{0}$ of \eqref{eq:nonlin} is \emph{stable in probability} if for any $\epsilon>0$
\[
\lim_{\mathbf{x}_0\to 0}\mathbb{P}\left[\sup_{t>0}|\mathbf{X}(t,\mathbf{x}_0)|>\epsilon\right]=0.
\]
It is \emph{asymptotically stable in probability} if it is stable in probability and
\[
\lim_{\mathbf{x}_0\to \mathbf{0}}\mathbb{P}\left[\lim_{t\to\infty}\mathbf{X}(t,\mathbf{x}_0)=0\right]=1
\]
It is \emph{asymptotically stable in the large} if it is stable in probability and for all $\mathbf{x}_0$
\begin{equation}\label{eq:asystabe}
\mathbb{P}\left[\lim_{t\to\infty}\mathbf{X}(t,\mathbf{x}_0)=0\right]=1.
\end{equation}
\end{definition}

\begin{definition}
The solution $\mathbf{X}(t)\equiv \mathbf{0}$ of \eqref{eq:nonlin} is \emph{mean-square stable} if 
\[
\lim_{\delta\to 0}\sup_{\substack{|\mathbf{x}_0|\leq \delta,\\t\geq 0}}\mathbb{E}|\mathbf{X}(t,\mathbf{x}_0)|^2=0.
\]
It is \emph{asymptotically mean-square stable} if it is mean-square stable and
\[
\lim_{t\to\infty}\mathbb{E}|\mathbf{X}(t,\mathbf{x}_0)|^2=0.
\]
It is \emph{exponentially mean-square stable} if for some constants $C$ and $\tilde\alpha$,
\[
\mathbb{E}|\mathbf{X}(t,\mathbf{x}_0)|^2\leq C|\mathbf{x}_0|^2e^{-\tilde\alpha t},\quad t\geq 0.
\]
\end{definition}
Then we have the following results.
\begin{theorem}[Khas'minskii~\cite{khas}]\label{thm:khas1}
If the equilibrium solution $\mathbf{B}(t)\equiv \mathbf{0}$ of the linear stochastic system \eqref{EQ:LINB} is asymptotically stable in the large, then the corresponding equilibrium solution  of the nonlinear stochastic system \eqref{eq:nonlin} is asymptotically stable in probability.
\end{theorem}
\begin{theorem}[Khas'minskii~\cite{khas}]\label{thm:khas2}
If all solutions of the linear stochastic system \eqref{EQ:LINB} satisfy
\begin{equation}\label{eq:asunst}
\mathbb{P}\left[\lim_{t\to\infty}|\mathbf{B}(t)|=\infty\right]=1,
\end{equation}
then the corresponding equilibrium of the nonlinear stochastic system \eqref{eq:nonlin} is unstable in probability.
\end{theorem}

\subsection{Exponential mean-square stability}
In \cite{FBR}, the authors use a result from Ryashko~\cite{Ryash} to show that the exponential mean-square stability of the equilibrium of \eqref{EQ:LINB} is given by the following result.
\begin{theorem}\label{thm:FBR}
The equilibrium $\mathbf{B}=0$ of system \eqref{EQ:LINB} is exponentially mean-square stable if and only if 
\begin{enumerate}
\item the equilibrium of the unperturbed system is asymptotically stable;
\item $tr(\mathbf{M}\mathbf{S})<1$, where $M$ is the second moments stationary matrix for the system with additive noise
\[
\begin{pmatrix}
dB_r\\ dB_\varphi
\end{pmatrix}=\begin{pmatrix}
-\varepsilon & -\delta\\
-g & -\varepsilon
\end{pmatrix}\begin{pmatrix}
B_r\\ B_\varphi
\end{pmatrix}
+\begin{pmatrix} \sqrt{2\sigma_1}\\ 0\end{pmatrix}dW_1,
\]
and $\mathbf{S}$ is given by
\[
S=\begin{pmatrix}
0 & 1\\
0 & 0
\end{pmatrix}.
\]
\end{enumerate}
\end{theorem}
\begin{remark}
For the drift-subcritical case (i.e. when $g\delta<\varepsilon^2$), Theorem \ref{thm:FBR} gives the following necessary and sufficient condition for exponential mean-square stability:
\[
\sigma_1<\frac{2\varepsilon(\varepsilon^2-g\delta)}{g^2}.
\]
If we take, for example, $g=0.99$, $\delta=0.01$, then the drift of \eqref{EQ:LINB} is subcritical if and only if $\varepsilon>0.09949874371$.
\end{remark}

Since we will compare exponential mean-square stability regions to regions of positive and negative $\lambda$ over a wide range of parameters, we make use of an alternative technique to compute such a condition which is not restricted to the subcritical case. In Buckwar \& Sickenberger~\cite{BuSi2}, it was shown that the system of first-order ordinary differential equations governing the time evolution of $\mathbb{E}\mathbf{B}(t)\mathbf{B}(t)^T$, where  $\mathbf{B}$ is  any solution of \eqref{EQ:LINB}, has coefficient
\begin{equation}\label{def:StabM}
 \mathcal{S} = \mathbb{I}_{d} \otimes \Lambda + \Lambda\otimes \mathbb{I}_{d} + \Sigma \otimes \Sigma \,.
\end{equation}
where $\otimes$ denotes the Kronecker product of two matrices.
Eq. \eqref{def:StabM} leads to the following result:
\begin{theorem}\label{thm:expMS}
The equilibrium of \eqref{EQ:LINB} will be exponentially mean-square stable if and only if $\tilde\alpha(S)<0$, where $\tilde\alpha(S)$ is the spectral abscissa of
\begin{equation}\label{eq:S}
\mathcal{S}=\begin{pmatrix}
-2\varepsilon & -\delta & -\delta & 2\sigma_1\\
-g & -2\varepsilon & 0 & -\delta\\
-g & 0 & -2\varepsilon & -\delta\\
0 & -g & -g & -2\varepsilon
\end{pmatrix}.
\end{equation}
\end{theorem}
\begin{remark}
The eigenvalues $\lambda_i$, $i=1,\ldots,4$ of \eqref{eq:S} are as follows:
\begin{eqnarray*}
\lambda_1&=&-2\,\varepsilon;\\ \lambda_2&=&\frac{1}{3}\,\sqrt 
[3]{54\,{g}^{2}\sigma+6\,\sqrt {81\,{g}^{4}{
\sigma}^{2}-48\,{g}^{3}{\delta}^{3}}}+{\frac {4 g \delta}{\sqrt [3]{54\,{g}^{
2}\sigma+6\,\sqrt {81\,{g}^{4}{\sigma}^{2}-48\,{g}^{3}{\delta}^{3}}}}
}-2\,\varepsilon;\\ 
\lambda_3&=&-\frac{1}{6}\,\sqrt [3]{54\,{g}^{2}\sigma+6
\,\sqrt {81\,{g}^{4}{\sigma}^{2}-48\,{g}^{3}{\delta}^{3}}}-{\frac 
{2 g \delta}{\sqrt [3]{54\,{g}^{2}\sigma+6\,\sqrt {81\,{g}^{4}{\sigma}^{2}-48\,
{g}^{3}{\delta}^{3}}}}}-2\,\varepsilon\\&&+
\frac{i\sqrt {3}}{6} \left(\sqrt [3]{54\,{g}^{2}\sigma+6\,\sqrt {81\,{g}^{4}{\sigma}^{2}-48\,{g}^{
3}{\delta}^{3}}}-{\frac {12 g 
\delta}{\sqrt [3]{54\,{g}^{2}\sigma+6\,\sqrt {81\,{g}^{4}{\sigma}^{2}-48\,{g}^{3}{\delta}^{3}
}}}} \right); \\ 
\lambda_4&=&-\frac{1}{6}\sqrt 
[3]{54\,{g}^{2}\sigma+6\,\sqrt {81\,{g}^{4}{
\sigma}^{2}-48\,{g}^{3}{\delta}^{3}}}-{\frac {2 g \delta}{\sqrt [3]{54\,{g}^{
2}\sigma+6\,\sqrt {-48\,{g}^{3}{\delta}^{3}+81\,{g}^{4}{\sigma}^{2}}}}
}-2\,\varepsilon\\&&-\frac{i\sqrt {3}}{6} \left(\sqrt [3]{54\,{g}^{2}\sigma+
6\,\sqrt {81\,{g}^{4}{\sigma}^{2}-48\,{g}^{3}{\delta}^{3}}}-{
\frac {12 g \delta}{\sqrt [3]{54\,{g}^{2}\sigma+6\,\sqrt 
{81\,{g}^{4}{\sigma}^{2}-48\,{g}^{3}{\delta}^{3}}}}} \right).
\end{eqnarray*}
While it is difficult to extract an exponential mean-square stability condition in closed form, we can easily compute precise stability regions using Maple.
\end{remark}
\begin{remark}
The asymptotic stability in the large of the linearised equilibrium in the statement of Theorem \ref{thm:khas1} can be replaced with exponential mean-square stability.
However for a necessary and sufficient condition we need to compute the upper Lyapunov exponent.
\end{remark}

\subsection{Upper Lyapunov exponents of linear stochastic systems}
An excellent account of Lyapunov exponents for two-dimensional linear stochastic systems type is given by Imkeller \& Lederer~\cite{IL}, who refer us to Arnold~\cite{ArnoldBook} for a more general treatment. We summarise the relevant ideas for this article in this section, but similarly refer the reader to \cite{ArnoldBook,IL} for more comprehensive exposition.

Consider the two-dimensional stochastic differential equation of Stratonovich type
\begin{equation}\label{eq:main}
d\mathbf{X}(t)=A_0\mathbf{X}(t)dt+A_1\mathbf{X}(t)\circ dW(t),\quad t\geq 0,
\end{equation}
where $A_{0},A_{1}\in\mathbb{R}^{2\times 2}$, $W$ is a scalar Wiener process, and the initial value $\mathbf{X}(0)=\mathbf{x}_0\in\mathbb{R}^2$ is deterministic. 

It was proved in \cite{Arnold2} that the asymptotic exponential growth rate of \eqref{eq:main} is non-random and equal to the (top or upper) Lyapunov exponent:
\begin{definition}\label{def:Lyap}
The (top or upper) Lyapunov exponent of the stochastic system \eqref{eq:main} is 
\[
\lambda=\lim_{t\to\infty}\frac{1}{t}\log|\mathbf{X}(t,\mathbf{x}_0)|,\quad a.s.
\]
\end{definition}
\begin{remark}
The equilibrium solution $\mathbf{X}(t)\equiv 0$ of \eqref{eq:main} is asymptotically stable in the large if and only if $\lambda<0$. Note also that if $\lambda>0$, $\mathbf{X}(t)\equiv \mathbf{0}$ is unstable in the specific sense of \eqref{eq:asunst}. 
\end{remark}

Taking the approach of \cite{Arnold1,Arnold2,Arnold3} to compute $\lambda$, project the non-equilibrium solutions of \eqref{eq:main} from $\mathbb{R}^2\setminus\{0\}$ onto the unit circle $S^0=\{v\in\mathbb{R}^2\,:\,|v|=1\}$ via the transformation $x\mapsto x/|x|$. The resulting process $S(t):=X(t)/|X(t)|$ satisfies the scalar SDE
\[
dS(t)=h_{A_0}(S(t))dt+h_{A_1}(S(t))\circ dW(t),\quad t\geq 0,
\]
on $S^0$, where $h_{A_0}(s)=A_0s-\langle s,A_0s\rangle$, similarly for $h_{A_1}$, and $\langle\cdot,\cdot\rangle$ denotes the inner product in $\mathbb{R}^2$. We can now use the relation $(s_1,s_2)^T=(\cos\varphi,\sin\varphi)^T$ to write the angular component of \eqref{eq:main} as the scalar SDE
\begin{equation}\label{eq:angmot}
d\varphi(t)=h_0(\varphi(t))dt+h_1(\varphi(t))\circ dW(t),\quad t\geq 0.
\end{equation}
for some functions $h_0,h_1$. When \eqref{eq:angmot} is sufficiently nondegenerate, in the sense that it satisfies Hormander's hypoellipticity condition (and this will be the case here), then solutions of \eqref{eq:angmot} have a unique stationary distribution forward in time on the unit circle, and the Lyapunov exponents may be expressed, via Birkhoff's ergodic theorem, as a spatial average over $[-\pi/2,\pi/2)$. For Eq. \eqref{eq:main}, the following form is given by \cite{IL}:

\begin{lemma}\label{lem:lyapfir}
The upper Lyapunov exponent of \eqref{eq:main} is given by 
\begin{equation}\label{eq:lamfir}
\lambda=\frac{\text{tr} A_0}{2}+\frac{1}{2}\int_{-\pi/2}^{\pi/2}\frac{h_1'(\varphi)h_0(\varphi)-h_0'(\varphi)h_1(\varphi)}{h_1(\varphi)}p(\varphi)d\varphi,
\end{equation}
where $p$ is the unique forward stationary distribution of \eqref{eq:angmot}.
\end{lemma}
For the specific drift structure encountered in this article, the expression \eqref{eq:lamfir} in the statement of Lemma \ref{lem:lyapfir} was refined further in \cite{IL}, and we will use that in the next section.

\section{Asymptotic properties of the model system}\label{asasy}
\subsection{Computing the upper Lyapunov exponents of \eqref{EQ:LINB}}
Note that if $W$ is a Wiener process and $2\sigma_1>0$, then $\tilde{W}(t)=-\sqrt{2\sigma_1} W(t/2\sigma_1)$ is also a Wiener process. Apply this transform, reverse the components of the system and, noting that the diffusion coefficient is nilpotent of degree 2, rewrite \eqref{EQ:LINB} as the Stratonovich SDE
\begin{equation}\label{eq:linBstrat}
\begin{pmatrix}
dB_{\tilde{\varphi}}\\ dB_{\tilde{r}}
\end{pmatrix}
=\begin{pmatrix}
-\frac{\varepsilon}{2\sigma_1} & -\frac{g}{2\sigma_1}\\
-\frac{\delta}{2\sigma_1} & -\frac{\varepsilon}{2\sigma_1}
\end{pmatrix}
\begin{pmatrix}
B_{\tilde{\varphi}}\\ B_{\tilde{r}}
\end{pmatrix}
dt+\begin{pmatrix}
0 & 0\\
1 & 0
\end{pmatrix}
\begin{pmatrix}
B_{\tilde{\varphi}}\\ B_{\tilde{r}}
\end{pmatrix}
\circ d\tilde{W}(t),\quad t\geq 0.
\end{equation}
The transformation from It\^o to Stratonovich (and vice-versa) may be found in \cite{Mao}. Now apply the following result:
\begin{lemma}[Imkeller \& Lederer~\cite{IL}, Theorem 2.1]\label{lem:lemIL}
Consider the 2-dimensional Stratonovich SDE \eqref{eq:main} where $A_0=\begin{pmatrix}a_{11}&a_{12}\\ a_{21}&a_{22}\end{pmatrix}$ is such that $a_{12}\neq 0$ and possesses eigenvalues $\mu_1,\mu_2$ with $\Re \mu_1>\Re \mu_2$. Moreover suppose that $A_1=\begin{pmatrix}\alpha & 0\\ 1 & \alpha\end{pmatrix}$. Then the Lyapunov exponents $\lambda_1,\lambda_2$ satisfy
\[
\lambda_{1/2}=\frac{\mu_1+\mu_2}{2}\pm\frac{1}{2}|a_{12}|\frac{\int_0^\infty\sqrt{v}\exp(-\frac{1}{6}|a_{12}|v^3+\frac{v}{2|a_{12}|}(\mu_1-\mu_2)^2)dv}{\int_0^\infty\frac{1}{\sqrt{v}}\exp(-\frac{1}{6}|a_{12}|v^3+\frac{v}{2|a_{12}|}(\mu_1-\mu_2)^2)dv}.
\]
\end{lemma}
\begin{remark}
Note that \cite{IL} use the non-standard matrix index notation
\[
A_0=\begin{pmatrix}a_{11}&a_{21}\\ a_{12}&a_{22}\end{pmatrix},
\]
where $a_{12}$ and $a_{21}$ are reversed. We have switched to standard notation here to avoid confusion. 
\end{remark}

\begin{lemma}\label{lem:rotref1}
The upper Lyapunov exponent of \eqref{EQ:LINB} is given by
\begin{equation}\label{eq:LEint}
\lambda=-\varepsilon+\frac{g}{2}\cdot\frac{\int_0^\infty\sqrt{v}\exp\left(-\frac{1}{12}\frac{g}{\sigma_1}v^3+\frac{\delta }{\sigma_1}v\right)dv}{\int_0^\infty\frac{1}{\sqrt{v}}\exp\left(-\frac{1}{12}\frac{g}{\sigma_1}v^3+\frac{\delta }{\sigma_1}v\right)dv}.
\end{equation}
\end{lemma}
\begin{proof}
Noting that the eigenvalues of the drift coefficient are given by
\[
\mu_1=\frac{-\varepsilon+\sqrt{\delta g}}{2\sigma_1},\qquad \mu_2=\frac{-\varepsilon-\sqrt{\delta g}}{2\sigma_1},
\]
and that the Lyapunov exponents $\lambda_i$ of \eqref{EQ:LINB} and $\tilde\lambda_i$ of \eqref{eq:linBstrat} are related by $\lambda_i=2\sigma_1\tilde\lambda_i$, for $i=1,2$. A direct application of Lemma \ref{lem:lemIL}  to the system \eqref{eq:linBstrat} gives the result.
\end{proof}

\begin{remark} 
The equilibrium solution of \eqref{EQ:LINB} therefore obeys \eqref{eq:asystabe} in Definition \ref{def:asstab} if and only if
\begin{equation}\label{eq:asstacritier}
-\varepsilon+\frac{g}{2}\cdot\frac{\int_0^\infty\sqrt{v}\exp\left(-\frac{1}{12}\frac{g}{\sigma_1}v^3+\frac{\delta }{\sigma_1}v\right)dv}{\int_0^\infty\frac{1}{\sqrt{v}}\exp\left(-\frac{1}{12}\frac{g}{\sigma_1}v^3+\frac{\delta }{\sigma_1}v\right)dv}<0.
\end{equation}
Again following the approach of \cite{IL99, IL}, we provide an expansion in terms of gamma and generalised hypergeometric functions that allows for the efficient computation of the upper Lyapunov exponent.
\end{remark}

\begin{lemma}\label{lem:LEgam}
The Lyapunov exponent given by \eqref{eq:LEint} may be expressed in the following form
\begin{equation}\label{eq:LEgam}
\lambda=-\varepsilon+\sqrt[3]{\frac{3\sigma_1 g^2}{2}}\cdot\frac{\sum_{n=0}^{\infty}\left[\left(\frac{b}{\sqrt[3]{a}}\right)^n\cdot\frac{1}{n!}\cdot\Gamma(\frac{n}{3}+\frac{1}{2})\right]}{\sum_{n=0}^{\infty}\left[\left(\frac{b}{\sqrt[3]{a}}\right)^n\cdot\frac{1}{n!}\cdot\Gamma(\frac{n}{3}+\frac{1}{6})\right]},
\end{equation}
where $a=g/12\sigma_1$ and $b=\delta/\sigma_1$. 
\end{lemma}
\begin{proof}
First, we treat the numerator of the second term on the right hand side of \eqref{eq:LEint}. Applying the change of variables $u=av^3$ to the numerator of \eqref{eq:asstacritier}, we have
\begin{eqnarray*}
\lefteqn{\int_{0}^{\infty}\sqrt{v}\exp\left(-\frac{g}{12\sigma_1}v^3+\frac{\delta}{\sigma_1}v\right)dv}\\&=&\int_0^{\infty}\sqrt{v}\exp\left(-av^3+bv\right)dv\\&=&\frac{1}{3\sqrt{a}}\int_{0}^{\infty}\frac{1}{\sqrt{u}}\exp\left(b\sqrt[3]{\frac{u}{a}}\right)\exp(-u)du\\
&=&\frac{1}{3\sqrt{a}}\int_{0}^{\infty}\frac{1}{\sqrt{u}}\sum_{n=0}^{\infty}\left[\left(\frac{b}{\sqrt[3]{a}}\right)^n\cdot\frac{1}{n!}\cdot u^{n/3}\right]\exp(-u)du\\
&=&\frac{1}{3\sqrt{a}}\sum_{n=0}^{\infty}\left[\left(\frac{b}{\sqrt[3]{a}}\right)^n\cdot\frac{1}{n!}\cdot\int_{0}^{\infty}u^{n/3-1/2}\exp(-u)du\right]\\
&=&\frac{1}{3\sqrt{a}}\sum_{n=0}^{\infty}\left[\left(\frac{b}{\sqrt[3]{a}}\right)^n\cdot\frac{1}{n!}\cdot\Gamma\left(\frac{n}{3}+\frac{1}{2}\right)\right],
\end{eqnarray*}
where we expand the first exponential term in the integrand around $u=0$ at the third step and invoke Lebesgue's dominated convergence theorem at the fourth step.

The same transformation may be applied to the denominator to get
\[
\int_{0}^{\infty}\frac{1}{\sqrt{v}}\exp\left(-\frac{g}{12\sigma_1}v^3+\frac{\delta}{\sigma_1}v\right)dv=\frac{1}{3a^{1/6}}\sum_{n=0}^{\infty}\left[\left(\frac{b}{\sqrt[3]{a}}\right)^n\cdot\frac{1}{n!}\cdot\Gamma\left(\frac{n}{3}+\frac{1}{6}\right)\right].
\]
Substituting back in to \eqref{eq:LEint} gives the statement of the lemma.
\end{proof}

The form of $\lambda$ given by \eqref{eq:LEgam} in Lemma \ref{lem:LEgam} lends itself to a representation using generalised hypergeometric functions of the form given in the following definition.
\begin{definition}[Generalised hypergeometric functions]\label{def:hypgeodef}
For $k,l\in\mathbb{N}$, $a_1,\ldots,a_k,b_1,\ldots,b_l\in\mathbb{R}\setminus\mathbb{Z}_-$ we define
\begin{equation}\label{eq:hypgeodef}
{}_kF_l(a_1,\ldots,a_k;b_1,\ldots,b_l;x)=\sum_{n=0}^{\infty}\frac{(a_1)_n\cdots (a_k)_n}{(b_1)_n\cdots (b_l)_n}\,\frac{x^n}{n!},\quad x\in I,
\end{equation}
where the Pochammer symbol $(\cdot)_n$ is defined
\[
(a)_n=\prod_{i=0}^{n-1}(a+i),\quad n\geq 0,
\]
or equivalently,
\begin{equation}\label{eq:bra}
(a)_n=\frac{\Gamma(a+n)}{\Gamma(a)},
\end{equation}
so that $(a)_0=1$. $I\subseteq\mathbb{R}$ is the interval, centred at $0$, upon which the right-hand-side of \eqref{eq:hypgeodef} converges. In particular, when $k<l+1$, $I=\mathbb{R}$.
\end{definition}
\begin{lemma}\label{lem:fac}
The following identities hold:
\begin{align}
\frac{n!}{(3n)!}=\frac{1}{3^{2n+1}(\frac{1}{3})_n(\frac{2}{3})_n};\,\,\frac{n!}{(3n+1)!}&=\frac{1}{3^{2n+1}(\frac{2}{3})_n(\frac{4}{3})_n};\nonumber\\ \frac{n!}{(3n+2)!}&=\frac{1}{2}\cdot\frac{1}{3^{2n+1}(\frac{4}{3})_n(\frac{5}{3})_n}.\label{eq:fac}
\end{align}
\end{lemma}

\begin{proof}
It follows from \eqref{eq:bra} that
\begin{eqnarray*}
\left(\frac{1}{3}\right)_n\left(\frac{2}{3}\right)_n&=&\frac{1\cdot 2\cdot 4\cdot 5\cdot 7\cdot 8\cdots (3n-2)(3n-1)}{3^{2n}}=\frac{(3n)!/3(n!)}{3^{2n}};\\
\left(\frac{2}{3}\right)_n\left(\frac{4}{3}\right)_n&=&\frac{2\cdot 4\cdot 5\cdot 7\cdot 8\cdot 10\cdots (3n-1)(3n+1)}{3^{2n}}=\frac{(3n+1)!/3(n!)}{3^{2n}};\\
\left(\frac{4}{3}\right)_n\left(\frac{5}{3}\right)_n&=&\frac{4\cdot 5\cdot 7\cdot 8\cdot 10\cdot 11\cdots (3n+1)(3n+2)}{3^{2n}}=\frac{(3n)!/3(n!)}{2\cdot 3^{2n}}.
\end{eqnarray*}
\end{proof}

We now represent $\lambda$ in terms of the parameters of the system, using gamma and hypergeometric functions:
\begin{theorem}\label{thm:lamrhohyp}
The Lyapunov exponent given by \eqref{eq:LEint} may be expressed in the following form
\begin{equation}\label{eq:lamrhoComp}
\lambda=-\varepsilon+\sqrt[3]{\frac{3\sigma_1 g^2}{2}}\cdot\frac{G_1(g,\sigma_1,\delta)}{G_2(g,\sigma_1,\delta)},
\end{equation}
where
\begin{multline*}
G_1(g,\sigma_1\delta)=\Gamma\left(\frac{1}{2}\right){}_1F_2\left(\frac{1}{2};\frac{1}{3},\frac{2}{3};\frac{2\delta^3}{9g\sigma_1}\right)+\delta\sqrt[3]{\frac{2}{g\sigma_1}}\Gamma\left(\frac{5}{6}\right){}_1F_2\left(\frac{5}{6};\frac{2}{3},\frac{4}{3};\frac{2\delta^3}{9g\sigma_1}\right)\\+\frac{\delta^2}{12}\sqrt[3]{\frac{4}{g^2\sigma_1^2}}\Gamma\left(\frac{1}{6}\right){}_1F_2\left(\frac{7}{6};\frac{4}{3},\frac{5}{3};\frac{2\delta^3}{9g\sigma_1}\right),
\end{multline*}
and 
\begin{multline}\label{eq:G2}
G_2(g,\sigma_1\delta)=\Gamma\left(\frac{1}{6}\right){}_1F_2\left(\frac{1}{6};\frac{1}{3},\frac{2}{3};\frac{2\delta^3}{9g\sigma_1}\right)+\delta\sqrt[3]{\frac{2}{g\sigma_1}}\Gamma\left(\frac{1}{2}\right){}_1F_2\left(\frac{1}{2};\frac{2}{3},\frac{4}{3};\frac{2\delta^3}{9g\sigma_1}\right)\\+\frac{\delta^2}{12}\sqrt[3]{\frac{4}{g^2\sigma_1^2}}\Gamma\left(\frac{5}{6}\right){}_1F_2\left(\frac{5}{6};\frac{4}{3},\frac{5}{3};\frac{2\delta^3}{9g\sigma_1}\right).
\end{multline}
\end{theorem}
\begin{proof}
Once again, we proceed by treating the numerator of the second term on the right hand side of \eqref{eq:LEgam}. Setting $A=b/\sqrt[3]{a}$, we can decompose the series as
\begin{multline*}
\sum_{n=0}^{\infty}\left[\left(\frac{b}{\sqrt[3]{a}}\right)^n\cdot\frac{1}{n!}\cdot\Gamma\left(\frac{n}{3}+\frac{1}{2}\right)\right]=\sum_{n=0}^{\infty}\left[A^n\cdot\frac{1}{n!}\cdot\Gamma\left(\frac{n}{3}+\frac{1}{2}\right)\right]\\
=\Gamma\left(\frac{1}{2}\right)\underbrace{\sum_{n=0}^{\infty}\left[\frac{(A^3)^n}{n!}\cdot\frac{n!}{(3n)!}\cdot \left(\frac{1}{2}\right)_n\right]}_{(a)}+A\Gamma\left(\frac{5}{6}\right)\underbrace{\sum_{n=0}^{\infty}\left[\frac{(A^3)^n}{n!}\cdot\frac{n!}{(3n+1)!}\cdot \left(\frac{5}{6}\right)_n\right]}_{(b)}\\+A^2\Gamma\left(\frac{7}{6}\right)\underbrace{\sum_{n=0}^{\infty}\left[\frac{(A^3)^n}{n!}\cdot\frac{n!}{(3n+2)!}\cdot \left(\frac{7}{6}\right)_n\right]}_{(c)}.
\end{multline*}
This decomposition arises due to the fact that, if we select only every third term of the series, the arguments of each gamma function will differ by an integer, and repeated iteration via $\Gamma(n+1)=n\Gamma(n)$ yields a collection of terms that differ from $\Gamma\left(\frac{1}{2}\right)$ only by a multiplicative constant. Taking out the common factor for terms with index 0 mod 3 leaves the sum labelled (a). Similarly, terms of the series with indices 1 (2) mod 3 differ from $A\Gamma\left(\frac{1}{2}\right)$ ($A^2\Gamma\left(\frac{7}{6}\right)$) by multiplicative constants which are gathered together to form (b) ((c)).

Now substitute \eqref{eq:fac} from Lemma \ref{lem:fac} to get
\begin{multline*}
\sum_{n=0}^{\infty}\left[A^n\cdot\frac{1}{n!}\cdot\Gamma\left(\frac{n}{3}+\frac{1}{2}\right)\right]\\=3^{-1}\Gamma\left(\frac{1}{2}\right)\sum_{n=0}^{\infty}\frac{(A^3/9)^n}{n!}\frac{(\frac{1}{2})_n}{(\frac{1}{3})_n(\frac{2}{3})_n}+\frac{A}{3}\Gamma\left(\frac{5}{6}\right)\sum_{n=0}^{\infty}\frac{(A^3/9)^n}{n!}\frac{(\frac{5}{6})_n}{(\frac{2}{3})_n(\frac{4}{3})_n}\\+\frac{A^2}{36}\Gamma\left(\frac{1}{6}\right)\sum_{n=0}^{\infty}\frac{(A^3/9)^n}{n!}\frac{(\frac{7}{6})_n}{(\frac{4}{3})_n(\frac{5}{3})_n}.
\end{multline*}
Applying Definition \ref{def:hypgeodef} yields
\begin{multline*}
\sum_{n=0}^{\infty}\left[A^n\cdot\frac{1}{n!}\cdot\Gamma\left(\frac{n}{3}+\frac{1}{2}\right)\right]\\=3^{-1}\Gamma\left(\frac{1}{2}\right){}_1F_2\left(\frac{1}{2};\frac{1}{3},\frac{2}{3};\frac{A^3}{9}\right)+\frac{A}{3}\Gamma\left(\frac{5}{6}\right){}_1F_2\left(\frac{5}{6};\frac{2}{3},\frac{4}{3};\frac{A^3}{9}\right)\\+\frac{A^2}{36}\Gamma\left(\frac{1}{6}\right){}_1F_2\left(\frac{7}{6};\frac{4}{3},\frac{5}{3};\frac{A^3}{9}\right).
\end{multline*}
A similar computation for the denominator of the second term on the right hand side of \eqref{eq:LEgam} yields
\begin{multline*}
\sum_{n=0}^{\infty}\left[A^n\cdot\frac{1}{n!}\cdot\Gamma\left(\frac{n}{3}+\frac{1}{6}\right)\right]\\=3^{-1}\Gamma\left(\frac{1}{6}\right){}_1F_2\left(\frac{1}{6};\frac{1}{3},\frac{2}{3};\frac{A^3}{9}\right)+\frac{A}{3}\Gamma\left(\frac{1}{2}\right){}_1F_2\left(\frac{1}{2};\frac{2}{3},\frac{4}{3};\frac{A^3}{9}\right)\\+\frac{A^2}{36}\Gamma\left(\frac{5}{6}\right){}_1F_2\left(\frac{5}{6};\frac{4}{3},\frac{5}{3};\frac{A^3}{9}\right).
\end{multline*}
Direct substitution back into \eqref{eq:LEgam} yields the statement of the theorem.
\end{proof}

\subsection{Stochastic stability of the equilibrium of the nonlinear model}
We may now provide conditions for the stability and instability in probability of the zero-equilibrium of \eqref{eq:nonlin}.
\begin{theorem}\label{thm:main}
The equilibrium solution $\mathbf{B}^\ast=\mathbf{0}$ of \eqref{eq:nonlin} is asymptotically stable in probability if $\lambda<0$,  and unstable in probability if $\lambda>0$, where $\lambda$ is the upper Lyapunov exponent of the linearised system \eqref{EQ:LINB} given by \eqref{eq:lamrhoComp} in Theorem \ref{thm:lamrhohyp}. 
\end{theorem}
\begin{proof}
The result is a consequence of Theorems \ref{thm:khas1} and \ref{thm:khas2}.
\end{proof}

\section{Visualising the asymptotic properties of the model system}\label{visualisation}
The forms of $\lambda$ and $\rho$ given in the statement of Theorem \ref{thm:lamrhohyp}, and $\tilde\alpha(\mathcal{S})$ in the statement of Theorem \ref{thm:expMS}, lend themselves to the use of mathematical software in order to visualise regions of stability. We use the {\tt GAMMA()} and {\tt hypergeom()} functions provided by Maple to generate plots of regions of positive and negative $\lambda$, where $\lambda$ is the upper Lyapunov exponent given by \eqref{eq:lamrhoComp} in Theorem \ref{thm:lamrhohyp}. We then superimpose the exponential mean-square stability regions over the regions of positive and negative $\lambda$ and compare the response to changes in the non-normality of the unperturbed system (as measured by \eqref{eq:Hnn} in Definition \ref{def:Hnn}).

In Figure \ref{fig0} we provide regions corresponding to positive and negative Lyapunov exponents in a subset of the $(\varepsilon,\sigma_1)$-plane, and observe the effect of increasing $\text{dep}_F(\Lambda)$ on the computed regions, by varying the value of $\delta\in[0.01,0.04]$ with $g=0.99$. Note that for fixed $\varepsilon,\delta$, increases in the intensity of the perturbation $\sigma_1$ have the effect of moving the Lyapunov exponent from negative to positive, but for fixed $(\varepsilon,\sigma_1)$ increasing the non-normality of the drift will move the sign of the upper Lyapunov exponent from positive back to negative. By Theorem \ref{thm:main}, the unshaded area corresponds to values of $(\varepsilon,\sigma_1)$ where the zero-equilibrium of the nonlinear model system \eqref{eq:nonlin} is unstable in probability, and the shaded area corresponds to values of $(\varepsilon,\sigma_1)$ where the zero-equilibrium of the nonlinear model system \eqref{eq:nonlin} is stable in probability: where catastrophic quenching has occurred. The linearised system \eqref{EQ:LINB} is drift-supercritical for all parameter values displayed here.

\begin{figure}
\begin{center}
$\begin{array}{c@{\hspace{0in}}c}
\mbox{\bf\small (a)}:\,\delta=0.04,\,\text{dep}_F(\Lambda)=0.95 & \mbox{\bf\small (b)}:\,\delta=0.03,\,\text{dep}_F(\Lambda)=0.96 \\
{\includegraphics[width=75mm,height=75mm]{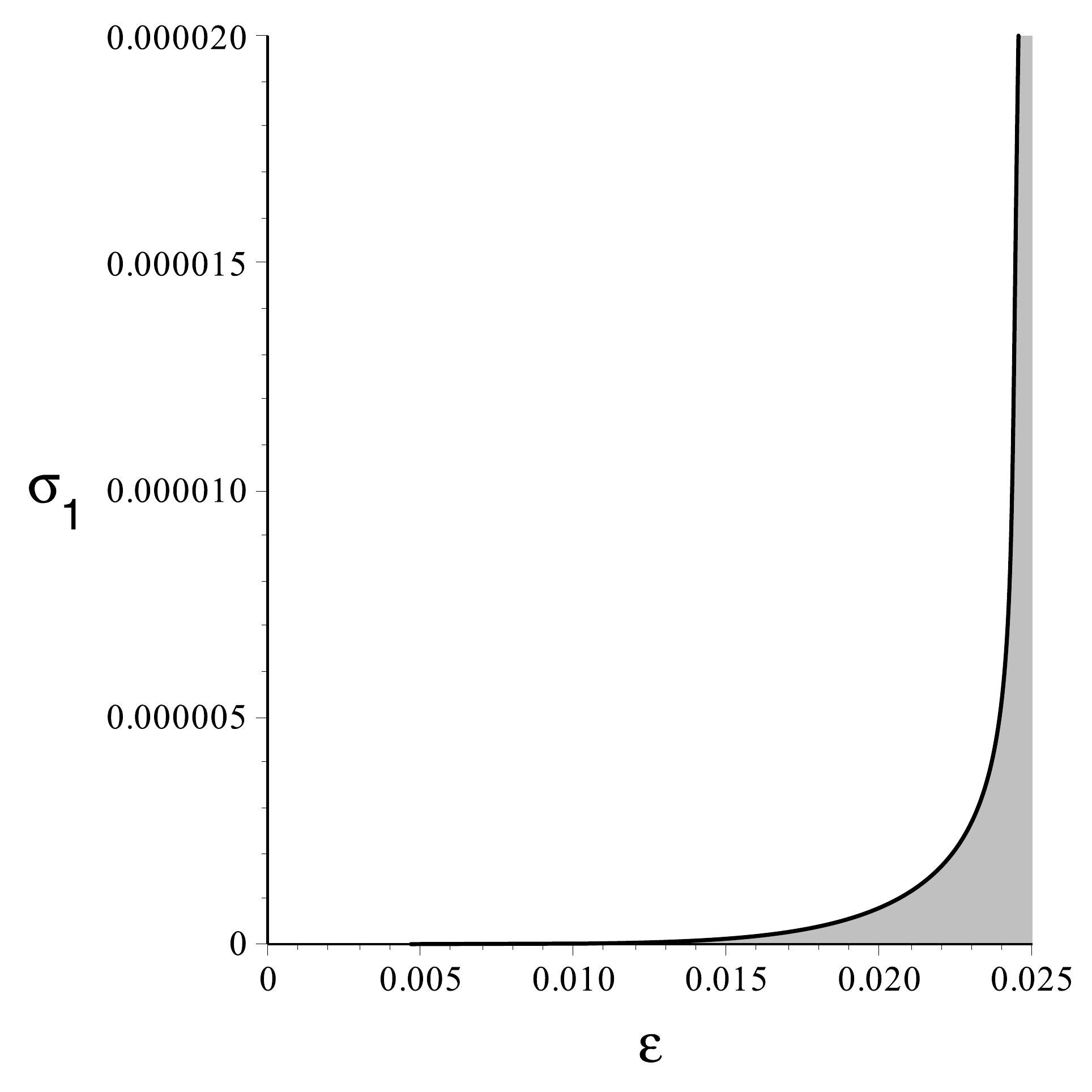}} & {\includegraphics[width=75mm,height=75mm]{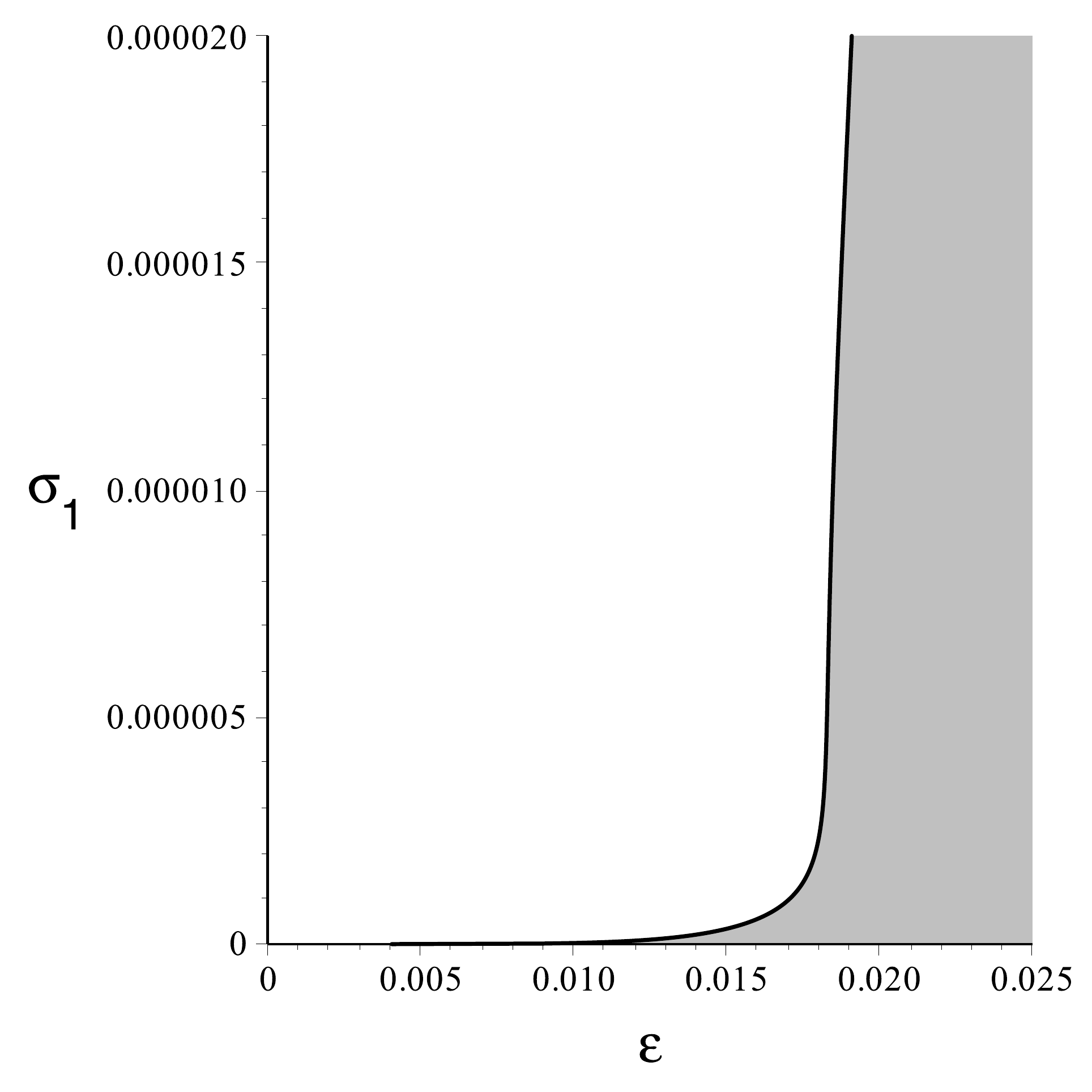}}\\ 
\mbox{\bf\small (c)}:\,\delta=0.02,\,\text{dep}_F(\Lambda)=0.97 & \mbox{\bf\small (d)}:\,\delta=0.01,\, \text{dep}_F(\Lambda)=0.98 \\
{\includegraphics[width=75mm,height=75mm]{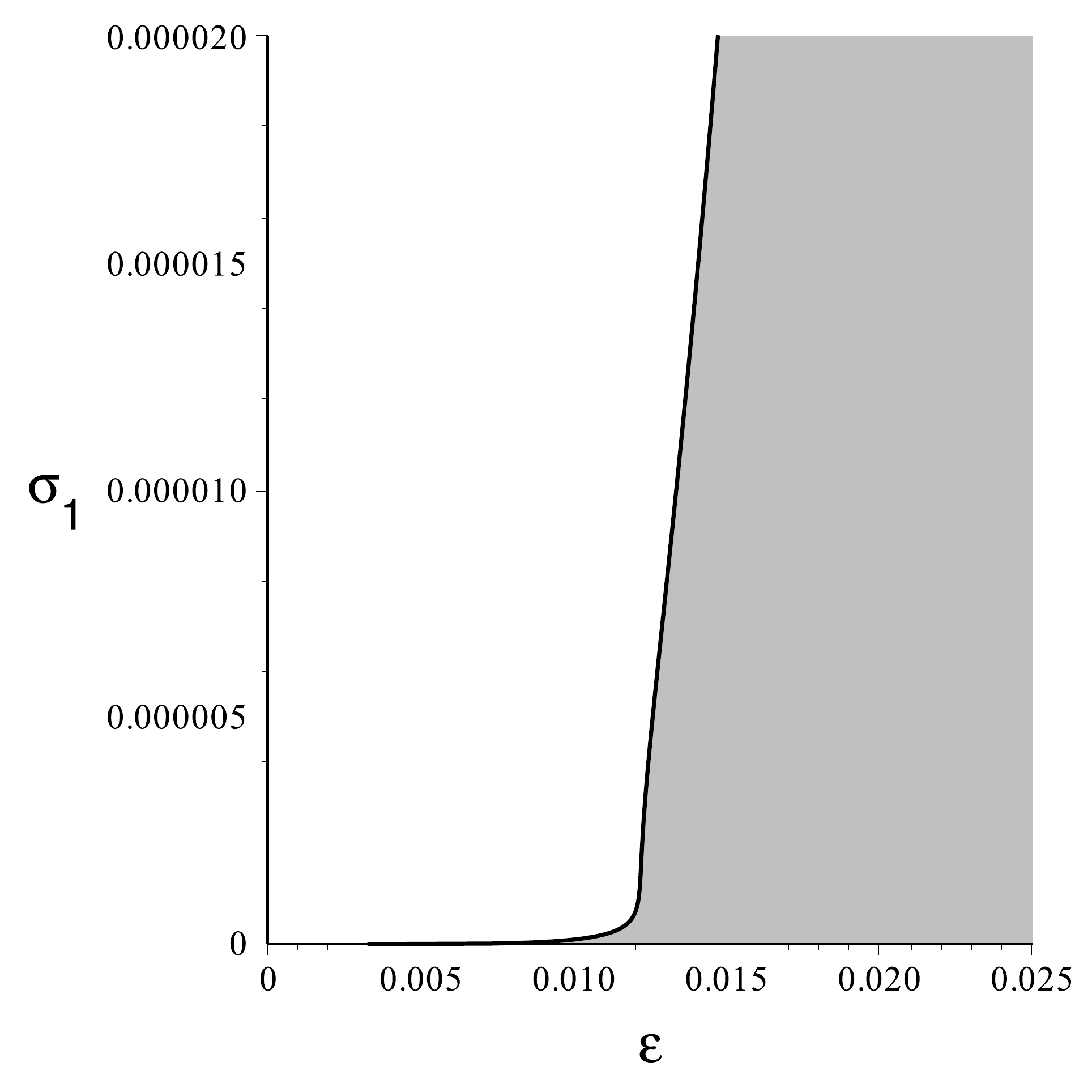}} & {\includegraphics[width=75mm,height=75mm]{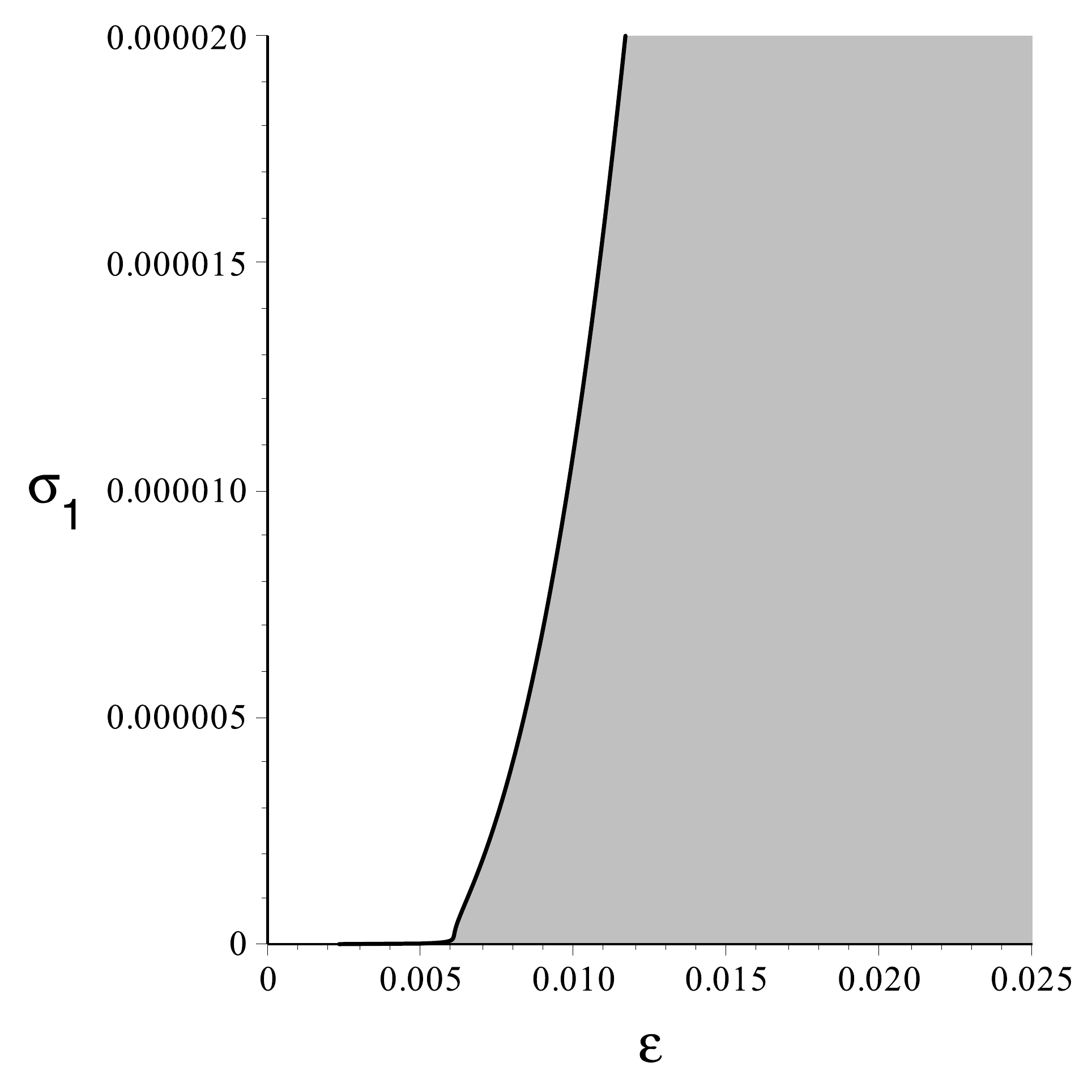}}
\end{array}$
\end{center}
\caption{The effect of varying the non-normality of the drift of \eqref{EQ:LINB} on the sign of the Lyapunov exponent $\lambda$ in the drift-supercritical regime ($g\delta>\varepsilon^2$). The solid line marks the boundary between regions where $\lambda>0$ (unshaded) and $\lambda<0$ (shaded). In each case, $g=0.99$.
}\label{fig0}
\end{figure}

In Figure \ref{fig2} we compare regions of exponential mean-square stability regions to regions of positive and negative $\lambda$ in a larger subset of the $(\varepsilon,\sigma_1)$-plane which includes both subcritical and supercritical regimes for the unperturbed equation. Note that, although all regions depend on $\delta$, the exponential mean-square stability region is more sensitive to changes in $\delta$ across this range of values.  By Theorem \ref{thm:main}, the zero-equilibrium of the nonlinear model system is stable in probability for parameter values in the shaded area, and unstable for parameter values in the unshaded area. Notice that a transition to instability occurs within the range $[0.01,0.1]$ on the horizontal axis, which was presented in \cite{Fed} as the typical range of values of $\varepsilon$ for this model. 

\begin{figure}
\begin{center}
$\begin{array}{c@{\hspace{0in}}c}
\mbox{\bf\small (a)}:\,\delta=0.04,\,\text{dep}_F(\Lambda)=0.95 & \mbox{\bf\small (b)}:\,\delta=0.03,\,\text{dep}_F(\Lambda)=0.96 \\
{\includegraphics[width=70mm,height=69mm]{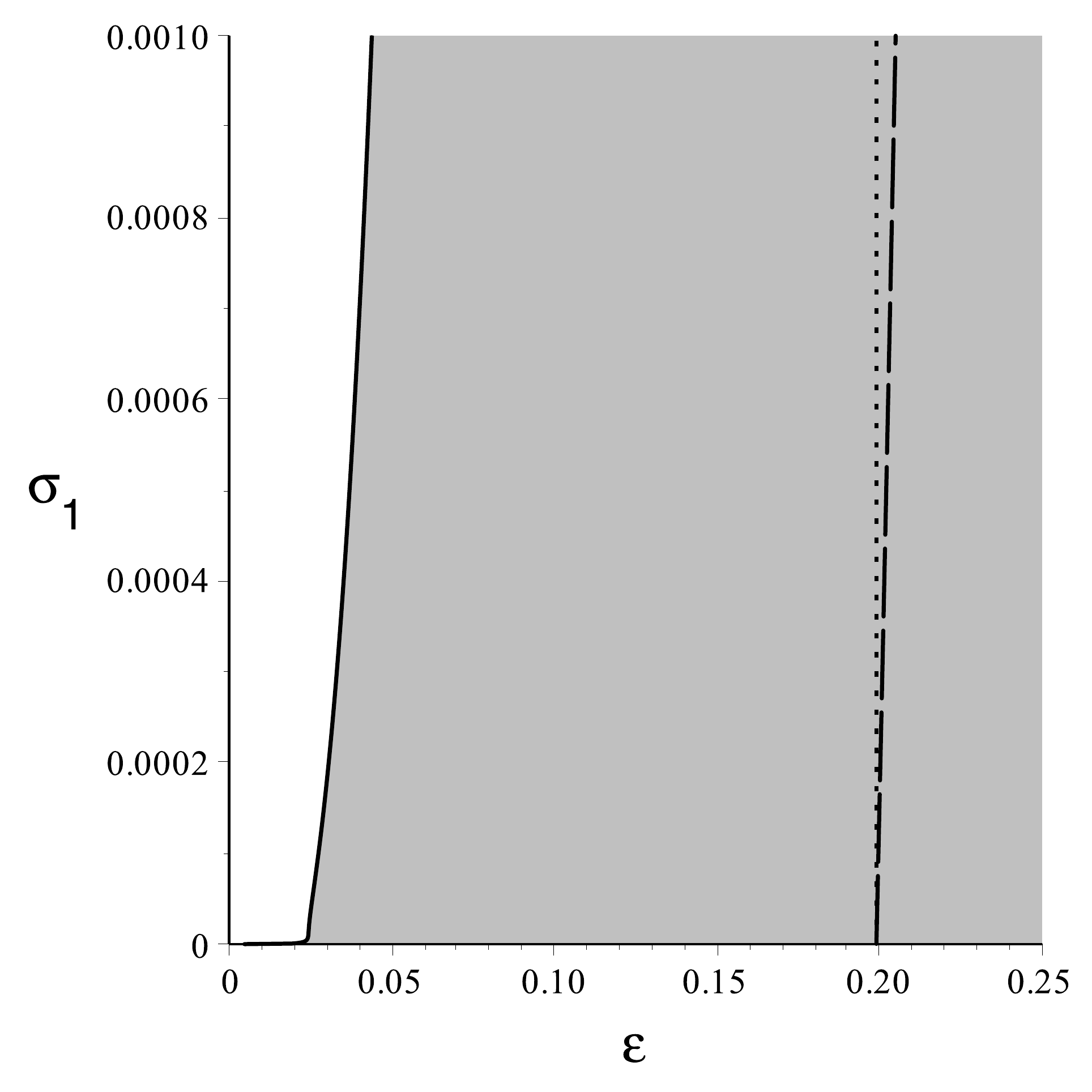}} & {\includegraphics[width=70mm,height=69mm]{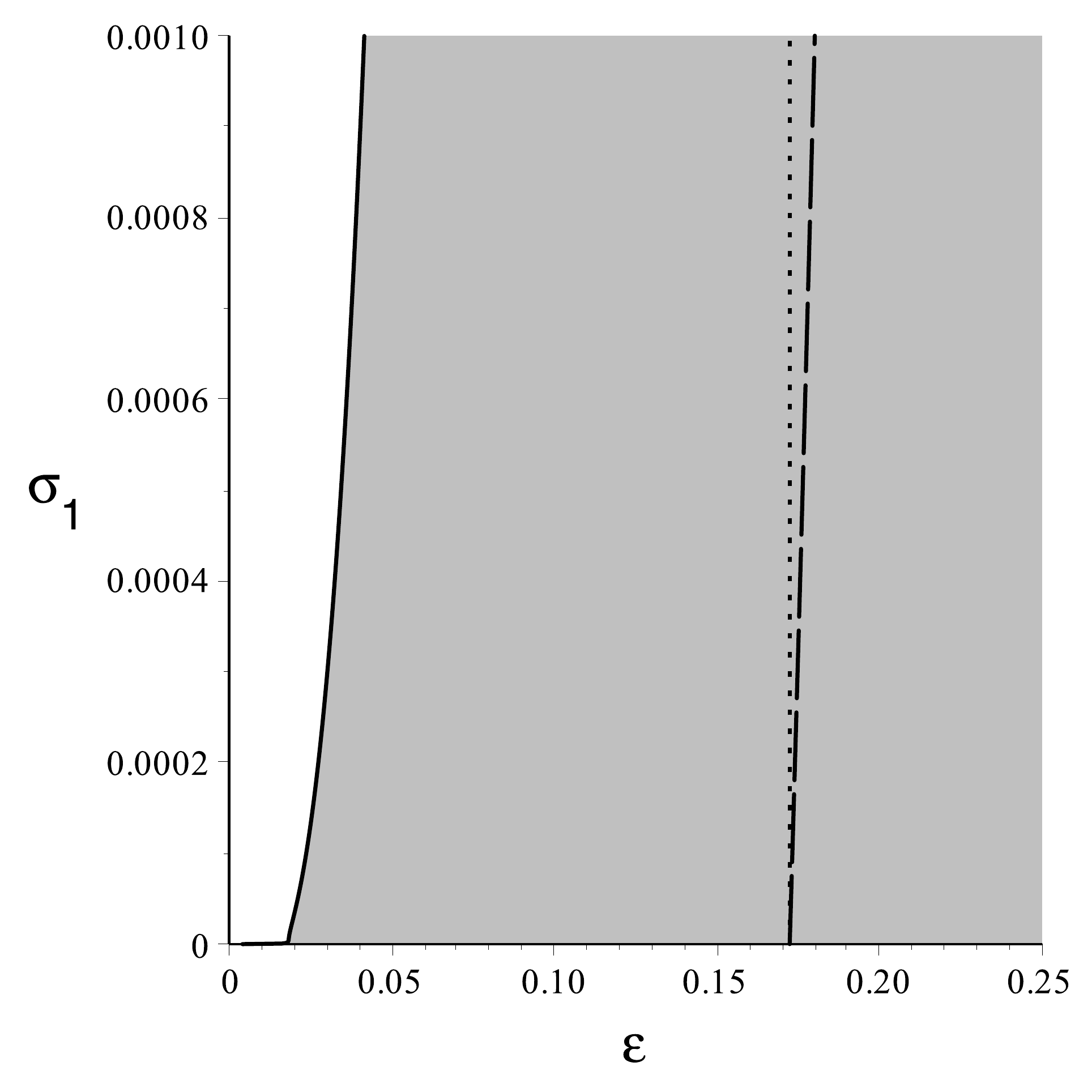}}\\ 
\mbox{\bf\small (c)}:\,\delta=0.02,\,\text{dep}_F(\Lambda)=0.97 & \mbox{\bf\small (d)}:\,\delta=0.01,\, \text{dep}_F(\Lambda)=0.98 \\
{\includegraphics[width=70mm,height=69mm]{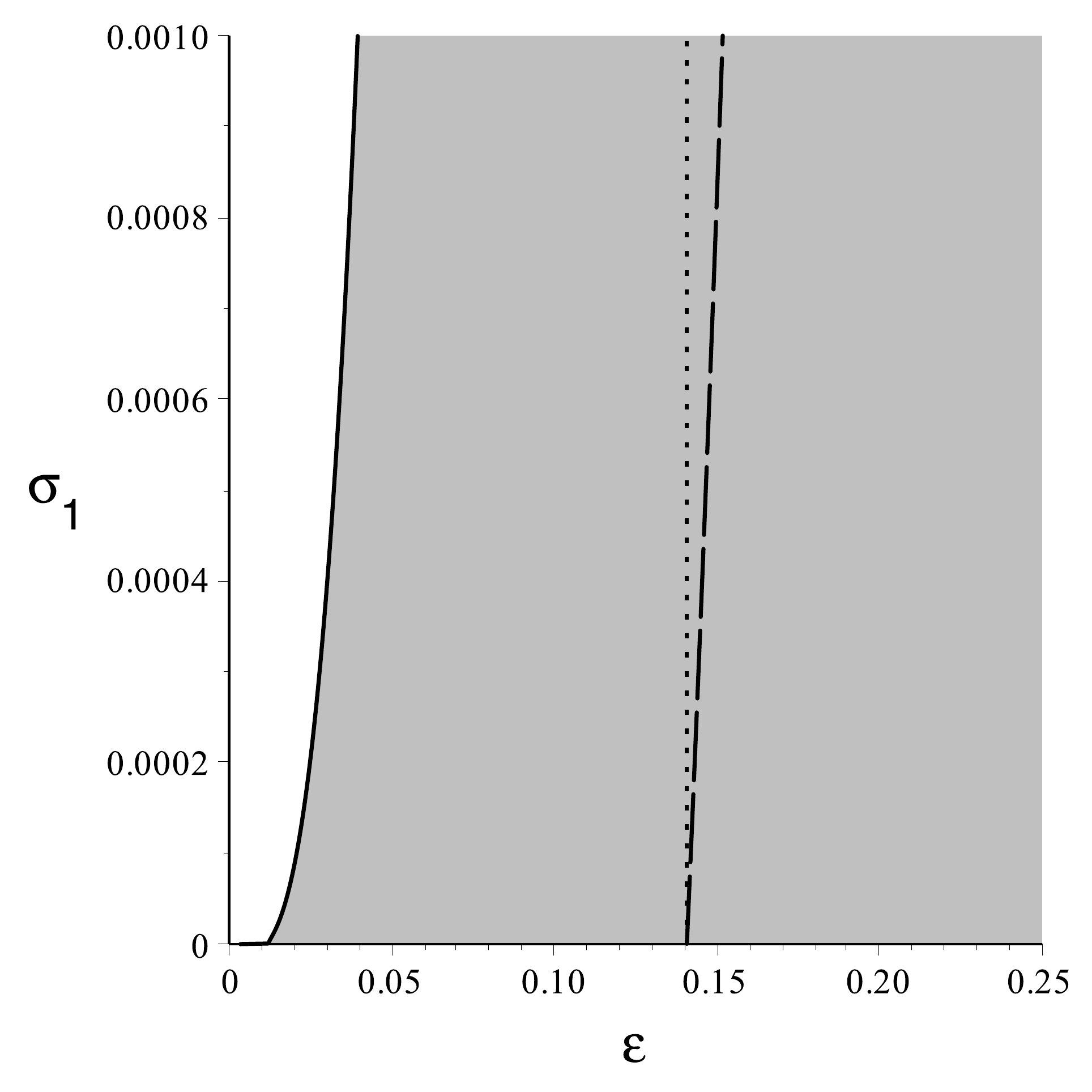}} & {\includegraphics[width=70mm,height=69mm]{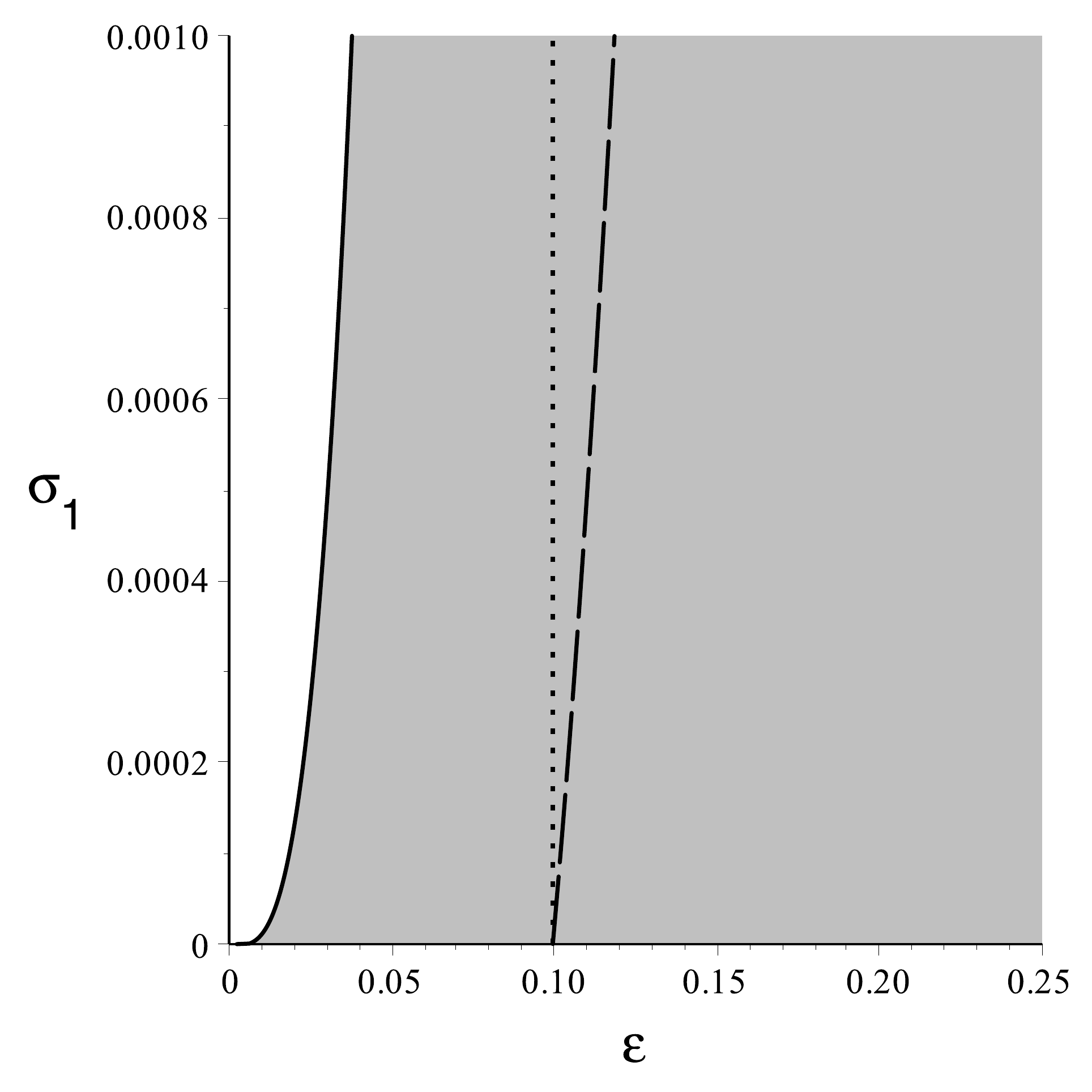}}
\end{array}$
\end{center}
\caption{The effect of varying the non-normality of the drift of \eqref{EQ:LINB} on regions of exponential mean-square stability, and on the sign of $\lambda$. The dotted line marks the boundary between the drift-supercritical regime (to the left) and drift-subcritical regime (to the right). The dashed line marks the boundary between regions of mean-square instability (to the left) and exponential mean-square stability (to the right) of \eqref{EQ:LINB}. The solid line marks the boundary between regions where $\lambda>0$ (unshaded) and $\lambda<0$ (shaded). In each case, $g=0.99$.
}\label{fig2}
\end{figure}

\section{Acknowledgments}
The author was supported by the London Mathematical Society Scheme 2 Grant ref. 21306, and a SQuaRE activity entitled ``Stochastic stabilisation of limit-cycle dynamics in ecology and neuroscience'' funded by the American Institute of Mathematics. 
He also wishes to thank Prof. Sergei Fedotov, The University of Manchester, and Prof. Peter Imkeller, Humboldt Universit\"at zu Berlin, for hosting him in June 2014 while he was on sabbatical, and for their advice and discussion.

\end{document}